\documentclass[10pt]{amsart}
\usepackage{amssymb}
\usepackage{bm}
\usepackage{graphicx}
\usepackage[centertags]{amsmath}
\usepackage{amsfonts}
\usepackage{amsthm}
\newtheorem{thm}{Theorem}
\newtheorem{cor}[thm]{Corollary}
\newtheorem{lem}[thm]{Lemma}

\newtheorem{prop}[thm]{Proposition}

\newtheorem{defn}[thm]{Definition}
\theoremstyle{definition}
\newtheorem{rem}{Remark}
\newtheorem{examp}{Example}

\newcommand{\ee}{\varepsilon}
\newcommand{\ccc}{\mathcal{C}}
\newcommand{\fff}{\mathcal{F}}

\newcommand{\con}{\smallfrown}
\newcommand{\incr}{[\ccc]^{2^n}_{\vartriangle}}
\newcommand{\ct}{2^{<\omega}}
\newcommand{\spl}{\mathrm{Spl}}

\begin{document}

\title{On filling families of finite subsets of the Cantor set}
\author{Pandelis Dodos and Vassilis Kanellopoulos}
\address{National Technical University of Athens, Faculty of Applied
Sciences, Department of Mathematics, Zografou Campus, 157 80,
Athens, Greece} \email{pdodos@math.ntua.gr, bkanel@math.ntua.gr}

\footnotetext[1]{2000 \textit{Mathematics Subject Classification}:
03E15, 05D10, 46B15.} \footnotetext[2]{Research supported by a
grant of EPEAEK program ``Pythagoras".}

\maketitle

%------------------------Abstract-------------------------------%

\begin{abstract}
Let $\ee>0$ and $\fff$ be a family of finite subsets of the Cantor
set $\ccc$. Following D. H. Fremlin, we say that $\fff$ is
$\ee$-filling over $\ccc$ if $\fff$ is hereditary and for every
$F\subseteq\ccc$ finite there exists $G\subseteq F$ such that
$G\in\fff$ and $|G|\geq\ee |F|$. We show that if $\fff$ is
$\ee$-filling over $\ccc$ and $C$-measurable in
$[\ccc]^{<\omega}$, then for every $P\subseteq\ccc$ perfect there
exists $Q\subseteq P$ perfect with $[Q]^{<\omega}\subseteq\fff$. A
similar result for weaker versions of density is also obtained.
\end{abstract}

%--------------------------Introduction-----------------------%

\section{Introduction}

Let $X$ be a set and $\ee>0$. A family $\fff\subseteq
[X]^{<\omega}$ is said to be $\ee$-filling over $X$ if $\fff$ is
hereditary (i.e. for every $F\in\fff$ and every $G\subseteq F$ we
have $G\in\fff$) and for every $F\in [X]^{<\omega}$ there exists
$G\subseteq F$ with $G\in\fff$ and $|G|\geq \ee |F|$. The notion
of an $\ee$-filling family is due to D. H. Fremlin \cite{F}, who
posed the following problem. For which cardinals $\kappa, \lambda$
we have that whenever $|X|=\kappa$ and $\fff\subseteq
[X]^{<\omega}$ is $\ee$-filling, then there exists $A\subseteq X$
with $|A|=\lambda$ and such that $[A]^{<\omega}\subseteq\fff$? It
is well-known that if $\kappa=\omega$, then $\lambda<\omega$. A
classical example is the Schreier family $\mathcal{S}=\{
F\subseteq \omega: |F|\leq\min F+1\}$. On the other hand, D. H.
Fremlin has shown (\cite{F}, Corollary 6D) that large cardinal
hypotheses imply the consistency of the statement that for every
$\ee$-filling family $\fff$ over $\mathfrak{c}$, there exists
$A\subseteq\mathfrak{c}$ infinite with
$[A]^{<\omega}\subseteq\fff$.

In this paper, we look at the problem when $X$ is the Cantor set
$\ccc=2^\omega$. Notice that $[\ccc]^{<\omega}$ has the structure
of a Polish space, being the direct sum of $[\ccc]^k$ $(k\geq 1)$.
S. A. Argyros, J. Lopez-Abad and S. Todor\v{c}evi\'{c} asked
whether the above mentioned result of Fremlin is valid
without extra set-theoretic assumptions provided that $\fff$ is
reasonably definable. We prove the following theorem which answers
this question positively.
\bigskip

\noindent \textbf{Theorem A.} \textit{Let $\fff$ be an
$\ee$-filling family over $\ccc$. If $\fff$ is $C$-measurable in
$[\ccc]^{<\omega}$, then for every $P\subseteq\ccc$ perfect there
exists $Q\subseteq P$ perfect with $[Q]^{<\omega}\subseteq\fff$. }
\bigskip

\noindent Actually we prove a more general result (Theorem
\ref{t1} in the main text) which implies, for instance, that
Theorem A is valid for an arbitrary $\ee$-filling family in the
Solovay Model.

Our second result concerns weaker versions of density. For every
$\fff\subseteq [\ccc]^{<\omega}$ and every $n\geq 1$ let
$d_{\fff}(n)$ be the density of $\fff$ at $n$, that is
\[ d_{\fff}(n)=\min_{F\in [\ccc]^n} \max\{|G|: G\subseteq F
\text{ and } G\in\fff\}. \]
Notice that $\fff$ is $\ee$-filling if and only if $\fff$ is
hereditary and $\frac{d_\fff(n)}{n}\geq\ee$ for all $n\geq 1$.
Although every $C$-measurable $\ee$-filling family $\fff$ over
$\ccc$ is not compact, Fremlin has shown that for every
$f:\omega\to\omega$ with $n\geq f(n)>0$ for all $n\geq 1$
and $\lim\frac{f(n)}{n}=0$ there exists a compact and hereditary
family $\fff$, closed in  $[\ccc]^{<\omega}$ and such
that $d_\fff(n)\geq f(n)$ for every $n\geq 1$ (see \cite{F},
Proposition 4B). The following theorem shows, however, that
any such family $\fff$ must still be large.
\bigskip

\noindent \textbf{Theorem B.} \textit{Let $\fff\subseteq
[\ccc]^{<\omega}$ hereditary. Assume that $\fff$ has the Baire
property in $[\ccc]^{<\omega}$ and satisfies
\begin{equation*}
(\ast) \ \ \ \ \ \limsup \frac{\log_2 d_\fff(2^n)}{\log_2 n}=+\infty.
\end{equation*}
Then for every $k\geq 1$ there exists $P\subseteq \ccc$ perfect
such that $[P]^k\subseteq\fff$. }
\bigskip

\noindent The proof of Theorem B is based on A. Blass' theorem
\cite{B}. Theorem B has the following consequence which shows that
we can increase the density of $\fff$ by passing to a perfect
subset. In particular, if $\fff$ is $C$-measurable and satisfies
equation $(\ast)$ above, then for every $f:\omega\to\omega$ with
$n\geq f(n)>0$ for all $n\geq 1$ and $\lim \frac{f(n)}{n}=0$ and
every perfect subset $P$ of $\ccc$ there exists $Q\subseteq P$
perfect such that the density of $\fff$ in $Q$ is greater or equal
to $f$. We also include some connections of the above results with
Banach spaces.
\bigskip

\noindent \textbf{Acknowledgments.} We would like to thank Professor
Spiros A. Argyros for bringing the problem to our attention as
well as for suggesting Corollary \ref{c3} and the Banach space
theoretic implications. We also thank Alexander D. Arvanitakis
for many stimulating conversations.

\section{Preliminaries}

We let $\omega=\{0,1,...\}$. The cardinality of a set $A$ is
denoted by $|A|$. By $<$ we denote the (strict) lexicographical
ordering on the Cantor set $\ccc=2^\omega$. If $A, B\subseteq
\ccc$, then we write $A<B$ if for every $x\in A$ and every $y\in
B$, we have $x<y$. For every $n\geq 1$ and every $P\subseteq
\ccc$, by $ [P]^n$ we denote the set of all $<$-increasing
sequences of $P$ of cardinality $n$, while by $[P]^{<\omega}$ the
set of all finite $<$-increasing sequences of $P$.

By $2^{<\omega}$ we denote the Cantor tree, i.e. the set of all
finite sequences of 0's and 1's, equipped with the (strict)
partial ordering $\sqsubset$ of initial segment. If $s,t\in
2^{<\omega}$, then by $s^\con t$ we denote their concatenation.
For every $s\in 2^{<\omega}$, the length $\ell(s)$ of $s$ is
defined to be the cardinality of the set $\{t\in 2^{<\omega}:
t\sqsubset s\}$. For every $n\in\omega$, by $2^n$ we denote the
set of all sequences in $2^{<\omega}$ of length $n$, while for
every $n\geq 1$ by $2^{<n}$ we denote the set of all sequences of
length less than $n$. For every $s,t\in 2^{<\omega}$ we denote by
$s\wedge t$ the $\sqsubset$-maximal node $w$ such that
$w\sqsubseteq s$ and $w\sqsubseteq t$. Similarly, if $x,y\in\ccc$,
then by $x\wedge y$ we denote the $\sqsubset$-maximal node $t$ of
$\ct$ with $t\sqsubset x$ and $t\sqsubset y$. We write $s\prec t$
if $w^\con 0\sqsubseteq s$ and $w^\con 1\sqsubseteq t$, where
$w=s\wedge t$.

We view every subset of $\ct$ as a \textit{subtree} of $\ct$
equipped with the induced partial ordering. For every $m\in\omega$
and every subtree $T$ of $\ct$ by $T(m)$ we denote the
$m$-\textit{level} of $T$, that is the set of all $t\in T$ such
that $|\{s\in T: s\sqsubset t\}|=m$. A node $t\in T$ is said to be
a \textit{splitting node} of $T$ if $t$ has at least two immediate
successors in $T$. By $\spl(T)$ we denote the set of splitting
nodes of $T$.

A subtree $T$ of $\ct$ is said to be \textit{downwards closed} if
for every $t\in T$ the set $\{s: s\sqsubseteq t\}$ is a subset of
$T$. Notice that if $T$ is a downwards closed subtree and
$m\in\omega$, then $T(m)=\{t\in T: t\in 2^m\}$. The body $[T]$ of
$T$ is the set $\{ x\in\ccc: x|n\in T \ \forall n\in\omega\}$,
where $x|n=(x_0,...,x_{n-1})\in 2^{<\omega}$ if $n\geq 1$ and
$x|0=(\varnothing)$ if $n=0$. If $t\in T$, then we set
$[T]_t=\{x\in [T]:t\sqsubset x\}$. In particular, for every
$t\in\ct$ we have $\ccc_t=\{ x\in \ccc: t\sqsubset x\}$.

If $A\subseteq\ct$, then the \textit{downwards closure} $\hat{A}$
of $A$ is the set $\{s\in\ct:\exists t\in A \text{ with }
s\sqsubseteq t\}$. Moreover, for every $F\subseteq\ccc$ we let
$T_F=\{ x|n: x\in F, n\in\omega\}$. Observe that $F$ is closed if
and only if $F=[T_F]$. It is easy to see that if $F$ is a finite
subset of $\ccc$, then $|\spl(T_F)|=|F|-1$. Similarly if $A$ is a
finite antichain of $\ct$, then $|\spl(\hat{A})|=|A|-1$.

A subtree $T$ of $\ct$ is said to be \textit{pruned} if for every
$t\in T$ there exists $s\in T$ with $t\sqsubset s$. It is said to
be \textit{skew} if for every $m\in\omega$ we have $|T(m)\cap
\spl(T)|\leq 1$.

Let us recall the notion of the \textit{type} $\tau$ of a
downwards closed, pruned, skew subtree $T$ of $\ct$ taken from
Louveau-Shelah-Veli\v{c}kovi\'{c} \cite{LSV}. We will only treat
trees $T$ with $[T]$ finite. So, let $k\geq 2$ and let $T$ be a
downwards closed, pruned, skew subtree of $\ct$ such that $[T]$
has $k$ elements. The type of $T$ is a function $\tau:\{ 1,...,
k-1\}\to\omega$, where $\tau(n)$ is defined as follows. For every
$n\in \{1,..., k-1\}$ let $m\in\omega$ be the least such that
$T(m)$ has $n+1$ nodes. Let $T(m-1)=\{ s_0\prec...\prec
s_{n-1}\}$. Then $\tau(n)=d$, if $s_d$ is the unique splitting
node of $T(m-1)$. Every type of a tree $T$ with $[T]=k$ will be
called a $k$-type. It is easy to see that for every $k\geq 2$
there exist $(k-1)!$ $k$-types. We remark that the above
definition is equivalent to the initial one, given by A. Blass
\cite{B}. If $F$ is a finite subset of $\ccc$, then we say that
$F$ is of type $\tau$ if $T_F$ is skew and of type $\tau$. If
$P\subseteq \ccc$ and $\tau$ is a $k$-type, then by $[P]^k_{\tau}$
we denote the set of all subsets of $P$ of type $\tau$.

We will also treat the following class of subtrees of $\ct$ which
are not downwards closed. A subtree $T$ of $\ct$ is said to be
\textit{regular dyadic} if $T$ can be written in the form
$T=(t_s)_{s\in 2^{<\omega}}$ such that for all $s_1, s_2\in\ct$
the following are satisfied.
\begin{enumerate}
\item[(1)] $s_1\sqsubset s_2$ (respectively $s_1\prec s_2$) if and
only if $t_{s_1}\sqsubset t_{s_2}$ (respectively $t_{s_1}\prec
t_{s_2}$). \item[(2)] $\ell(s_1)=\ell(s_2)$ if and only if
$\ell(t_{s_1})=\ell(t_{s_2})$.
\end{enumerate}
It is easy to see that the representation of $T$ as
$(t_s)_{s\in\ct}$ is unique. In what follows when we deal with a
regular dyadic subtree $T$ we will always use this unique
representation. We also notice that if $T$ is a regular dyadic
tree, then $[\hat{T}]$ is a perfect subset of $\ccc$ homeomorphic
to $\ccc$.

Finally, we recall that a subset $A$ of an uncountable Polish
space $X$ is $C$-\textit{measurable} if it belongs to the smallest
$\sigma$-algebra which is closed under the Souslin operation and
contains the open sets. We remark that the class of $C$-measurable
sets is strictly bigger than the $\sigma$-algebra generated by the
analytic sets (see \cite{Kechris}).

%--------------Definable $\ee$-filling families-------------%

\section{Definable $\ee$-filling families}

We start with the following definition.
\begin{defn}
Let $\fff\subseteq[\ccc]^{<\omega}$. The family $\fff$ is said to
have the Galvin property if for every $n\in\omega$ and every $P_0<
...< P_n$ perfect subsets of $\ccc$, there exist $Q_0,...,Q_n$
such that the following hold.
\begin{enumerate}
\item[(1)] For all $i=0,...,n$, $Q_i$ is a perfect subset of
$P_i$. \item[(2)] Either $Q_0\times ...\times Q_n \subseteq\fff$
or $(Q_0\times ...\times Q_n) \cap \fff =\varnothing$.
\end{enumerate}
\end{defn}
We notice that if for every $n\in\omega$ and every $P_0< ...< P_n$
perfect subsets of $\ccc$ the set $\fff\cap (P_0\times ...\times
P_n)$ has the Baire property in $P_0\times ...\times P_n$, then
the family $\fff$ has the Galvin property. This is a consequence
of a theorem of F. Galvin (see \cite{Kechris}, Theorem 19.6).
Under the above terminology we have the following.
\begin{thm}
\label{t1} Let $\ee>0$ and $\fff$ be an $\ee$-filling family over
$\ccc$. If $\fff$ has the Galvin property, then for every perfect
subset $P$ of $\ccc$ there exists $Q\subseteq P$ perfect such
that $[Q]^{<\omega}\subseteq \fff$.
\end{thm}
For the proof of Theorem \ref{t1} we need the following
definition.
\begin{defn}
\label{d1} Let $\fff\subseteq[\ccc]^{<\omega}$ and $T=(t_s)_{s\in
2^{<\omega}}$ be a regular dyadic subtree of $\ct$. We say that
the tree $T$ decides for $\fff$ if for every $n\in\omega$, every
$0\leq d\leq 2^n-1$ and every $F=\{s_0\prec ...\prec
s_d\}\subseteq 2^n$ we have that the product
$[\hat{T}]_{t_{s_0}}\times ... \times [\hat{T}]_{t_{s_d}}$ either
is included in or is disjoint from $\fff$. In the case where
$[\hat{T}]_{t_{s_0}}\times ... \times [\hat{T}]_{t_{s_d}}$ is
included in $\fff$, then we say that $F$ is trapped in $\fff$.
\end{defn}
The following lemma is the combinatorial part of the proof of
Theorem \ref{t1}.
\begin{lem}
\label{l1} Let $\fff\subseteq [\ccc]^{<\omega}$ with the Galvin
property and $P$ be a perfect subset of $\ccc$. Then there exists
a regular dyadic tree $T=(t_s)_{s\in 2^{<\omega}}$ that decides
for $\fff$ and $[\hat{T}]\subseteq P$.
\end{lem}
\begin{proof}
By recursion on the length of $s\in 2^{<\omega}$ we will build a
regular dyadic tree $T=(t_s)_{s\in 2^{<\omega}}$ and a family
$(P^s)_{s\in 2^{<\omega}}$ of subsets of $\ccc$ such that for all
$n\in\omega$ the following are satisfied.
\begin{enumerate}
\item[(1)] For all $s\in 2^n$, $P^s$ is a perfect subset of $P$.
\item[(2)] If $n\geq 1$, then for all $s\in 2^{n-1}$ and
$i\in\{0,1\}$, $P^{s^\con i}\subseteq P^s \cap \ccc_{t_{s^\con
i}}$. \item[(3)] For every $0\leq d\leq 2^n-1$ and every
$\{s_0\prec ...\prec s_d\}\subseteq 2^n$ we have that
$P^{s_0}\times ... \times P^{s_d}$ either is included in or is
disjoint from $\fff$.
\end{enumerate}
We start the construction. For $n=0$, we set
$t_\varnothing=\varnothing$. By the Galvin property of $\fff$,
there exists $P^\varnothing\subseteq P$ perfect such that either
$[P^\varnothing]^1 \subseteq\fff$ or
$[P^\varnothing]^1\cap\fff=\varnothing$. Then (1) and (3) are
satisfied. Now assume that for some $n\in\omega$, $(t_s)_{s\in
2^n}$ and $(P^s)_{s\in 2^n}$ have been constructed. As the family
$\{P^s:s\in 2^n\}$ consists of perfect subsets of $P$ and
$P^s\subseteq \ccc_{t_s}$, we may select a sequence $(t_s)_{s\in
2^{n+1}}$ such that the following are satisfied.
\begin{enumerate}
\item[(i)] For all $s_1, s_2\in 2^{n+1}$,
$\ell(t_{s_1})=\ell(t_{s_2})$. \item[(ii)] For every $s\in 2^n$,
the nodes $t_{s^\con 0}$ and $t_{s^\con 1}$ are successors of
$t_s$ and $t_{s^\con 0}\prec t_{s^\con 1}$. \item[(iii)] For all
$s\in 2^{n}$ and $i\in\{0,1\}$, setting $Q^{s^\con i}=P^s\cap
\ccc_{t_{s^\con i}}$ we have that $Q^{s^\con i}$ is a perfect
subset of $P^s$.
\end{enumerate}
Using the fact that $\fff$ has the Galvin property, by an
exhaustion argument over all subsets of $2^{n+1}$, we find for all
$s\in 2^{n+1}$ a perfect set $P^s\subseteq Q^s$ such that
condition (3) is satisfied. This completes the recursive
construction.

We will check that $T=(t_s)_{s\in 2^{<\omega}}$ satisfies all the
desired properties. First we observe that $[\hat{T}]\subseteq P$
is an immediate consequence of (1), (2) and the fact that
$P^\varnothing\subseteq P$. By (2) we also have that
$[\hat{T}]_{t_s}\subseteq P^{s}$ for all $s\in 2^{<\omega}$.
Hence, by (3) we get that $T$ decides for $\fff$, as desired.
\end{proof}
\begin{lem}
\label{l2} Let $\fff\subseteq[\ccc]^{<\omega}$ and $T=(t_s)_{s\in
2^{<\omega}}$ a regular dyadic tree that decides for $\fff$.
Assume that $\fff$ is $\ee$-filling for some $\ee>0$. Then the
following hold.
\begin{enumerate}
\item[(1)] For every $n\in\omega$, there exists $F_n\subseteq 2^n$
with $|F_n|\geq \ee \cdot 2^n$ and such that $F_n$ is trapped in
$\fff$. \item[(2)] Let $n,k\in\omega$ with $k\leq n$, $F\subseteq
2^n$ and $G\subseteq 2^k$ such that $G$ is dominated by $F$ (i.e.
for every $w\in G$ there exists $s\in F$ with $w\sqsubseteq s$).
If $F$ is trapped in $\fff$, then so does $G$.
\end{enumerate}
\end{lem}
\begin{proof}
(1) For every $s\in 2^n$, pick $x_s\in [\hat{T}]_{t_s}$. As $\fff$
is $\ee$-filling, there exists $F_n=\{s_0\prec ... \prec
s_{d-1}\}\subseteq 2^n$ with $d\geq \ee\cdot 2^n$ and such that
$\{x_s:s\in F_n\}\subseteq \fff$. It follows that
$([\hat{T}]_{t_{s_0}}\times ... \times [\hat{T}]_{t_{s_{d-1}}})
\cap \fff\neq\varnothing$. Since the tree $T$ decides for $\fff$,
we conclude that $F_n$ is trapped in $\fff$.\\
(2) First we notice that if $F$ is trapped in $\fff$, then every
subset of $F$ is also trapped in $\fff$, as $\fff$ is hereditary.
Now let $G$ be dominated by $F$. There exists $F'$ subset of $F$
with $|F'|=|G|$ and such that for every $w\in G$ there exists a
unique $s\in F'$ with $w\sqsubseteq s$. Arguing as in (1) above we
get that $G$ is trapped in $\fff$, as desired.
\end{proof}
For every regular dyadic tree $T=(t_s)_{s\in 2^{<\omega}}$ we
define a canonical Borel probability measure $\mu_T$ on
$[\hat{T}]$ by assigning to every $[\hat{T}]_{t_s}$, with $s\in
2^n$ and $n\in\omega$, measure equal to $\frac{1}{2^n}$. That is,
$\mu_T$ is the image of the usual measure on $\ccc$ induced by
the natural homeomorphism between $\ccc$ and $[\hat{T}]$. We
remark that $\mu_T$ is continuous (i.e. it vanishes on singletons)
and regular. The final lemma consists of the analytic part of the
argument.
\begin{lem}
\label{l3} Let $\fff\subseteq[\ccc]^{<\omega}$ and $T=(t_s)_{s\in
2^{<\omega}}$ a regular dyadic tree that decides for $\fff$.
Assume that $\fff$ is $\ee$-filling for some $\ee>0$. Then there
exists $K\subseteq [\hat{T}]$ closed such that $\mu_T(K)\geq\ee$
and $[K]^{<\omega} \subseteq \fff$.
\end{lem}
\begin{proof}
By Lemma \ref{l2}(1), for every $n\in\omega$ there exists
$F_n\subseteq 2^n$ with $|F_n|\geq\ee\cdot 2^n$ and such that
$F_n$ is trapped in $\fff$. Define
\[ C_n=\bigcup_{s\in F_n} [\hat{T}]_{t_s}. \]
Then $C_n$ is a clopen subset of $[\hat{T}]$ and moreover
$\mu_T(C_n)\geq\ee$ for every $n\in\omega$. Let us denote by
$\mathcal{K}([\hat{T}])$ the hyperspace of all compact subsets of
$[\hat{T}]$ equipped with the Vietoris topology. It is a compact
metrizable space (see \cite{Kechris}). Hence, there exist an
infinite subset $L$ of $\omega$ and $K\in \mathcal{K}([\hat{T}])$
such that the sequence $(C_n)_{n\in L}$ is convergent to $K$. As
the measure $\mu_T$ is regular, the map $\mathcal{K}([\hat{T}])\ni
K\mapsto \mu_T(K)$ is upper semicontinuous. It follows that
\[ \mu_T(K) \geq \limsup_{n\in L} \mu_T(C_n) \geq \ee. \]
It remains to show that $[K]^{<\omega}\subseteq \fff$. Indeed, let
$\{x_0< ... < x_l\} \subseteq K$. Since $K\subseteq [\hat{T}]$,
there exist $k\in\omega$ and $\{ w_0\prec ... \prec w_l\}\subseteq
2^k$ such that $t_{w_i}\sqsubset x_i$ for all $i=0,...,l$ (notice
that $\ell(t_{w_0})=...=\ell(t_{w_l})$). The sequence $(C_n)_{n\in
L}$ converges to $K$ and so there exists $n_0\in L$ such that for
all $n\in L$ with $n\geq n_0$, the set $\{t_s: s\in F_{n}\}$
dominates the set $\{t_{w_0},...,t_{w_l}\}$. The tree $T$ is
regular dyadic and so $F_n$ dominates $\{w_0,...,w_l\}$. As every
$F_n$ is trapped in $\fff$, by Lemma \ref{l2}(2) we get that
$\{w_0,...,w_l\}$ is trapped in $\fff$ too. This clearly implies
that $\{x_0,...,x_l\}\in\fff$ and the proof is completed.
\end{proof}
\begin{proof}[Proof of Theorem \ref{t1}]
Let $P\subseteq \ccc$ perfect. As $\fff$ has the Galvin property,
by Lemma \ref{l1} there exists a regular dyadic tree $T$ such that
$T$ decides for $\fff$ and $[\hat{T}]\subseteq P$. Since $\fff$ is
$\ee$-filling, by Lemma \ref{l3} there exists $K\subseteq
[\hat{T}]$ closed with $\mu_T(K)\geq\ee$ and such that
$[K]^{<\omega}\subseteq\fff$. As $\mu_T$ is continuous, $K$ is an
uncountable closed subset of $P$ and the result follows.
\end{proof}
\noindent \textbf{Consequences.} We notice that for every Polish
space $X$, every closed subset $F$ of $X$ and every $C$-measurable
subset $A$ of $X$, the set $A\cap F$ is $C$-measurable in $F$.
Invoking the classical fact that every $C$-measurable subset of a
Polish space has the Baire property (hence, by the remarks at the
beginning of the section, the Galvin property too), we get the
following corollary of Theorem \ref{t1}.
\begin{cor}
\label{c1} Let $\fff\subseteq [\ccc]^{<\omega}$ be $\ee$-filling.
If $\fff$ is $C$-measurable in $[\ccc]^{<\omega}$, then for every
$P\subseteq \ccc$ perfect, there exists $Q\subseteq P$ perfect
with $[Q]^{<\omega}\subseteq\fff$.
\end{cor}
As Projective Determinacy (PD) implies that every projective set
in a Polish space has the Baire property (see \cite{Kechris},
Theorem 38.17), under PD, Corollary \ref{c1} is also true for
every projective set.

There are some natural limitations on the possibility of extending
Corollary \ref{c1} for an arbitrary $\ee$-filling family. Indeed,
let $B$ be a Bernstein set, that is a subset of $\ccc$ such that
neither $B$ nor $\ccc\setminus B$ contain a perfect set. Setting
$\fff=[B]^{<\omega}\cup [\ccc\setminus B]^{<\omega}$ we see that
$\fff$ is $1/2$-filling, yet there does not exist a perfect set
$P$ with $[P]^{<\omega}\subseteq \fff$. Notice however that the
above counterexample is depended on the Axiom of Choice. As a
matter of fact, every counterexample known to us depends on the
Axiom of Choice. This is not an accident. As in the proof of
Theorem \ref{t1} we made no use of the Axiom of Choice, we have
the following corollary.
\begin{cor}
\label{c2} Assume ZF+DC and the statement that ``every subset of a
Polish space has the Baire property". Then for every $\ee$-filling
family $\fff$ over $\ccc$ and every $P\subseteq \ccc$ perfect,
there exists $Q\subseteq P$ perfect such that
$[Q]^{<\omega}\subseteq \fff$.
\end{cor}
We notice that the hypotheses of Corollary \ref{c2} hold in the
Solovay Model \cite{Sol} (see also \cite{BK} section 5.3, for a
discussion about this in a different but related context). A
similar result has been also obtained by A. W. Apter and M.
D\v{z}amonja \cite{AD}.
\begin{rem}
It follows by Corollary \ref{c1}, that if $\fff\subseteq
[\ccc]^{<\omega}$ is analytic and $\ee$-filling, then $\fff$
cannot be compact; that is there exists $A\subseteq \ccc$ infinite
such that $[A]^{<\omega}\subseteq\fff$. We should point out that
this can also be derived by the results of D. H. Fremlin in \cite{F}.
To see this, one argues by contradiction. So, assume that
$\fff\subseteq [\ccc]^{<\omega}$ is analytic, compact and
$\ee$-filling for some $\ee>0$. It was observed by S. A. Argyros,
J. Lopez-Abad and S. Todor\v{c}evi\'{c} that the rank of $\fff$ is
a countable ordinal whenever $\fff$ is analytic and compact. This
follows by a standard application of the Kunen-Martin theorem (see
\cite{Kechris}, Theorem 31.1). By Lemma 2C in \cite{F} applied to
the ideal $\mathcal{I}$ of countable subsets of $\ccc$, we get
that the rank of $\fff$ must be greater or equal to $\omega_1$,
which is a contradiction.
\end{rem}
\begin{rem}
By modifying the proof of Theorem \ref{t1} we have the following
result for an arbitrary family $\fff$.
\begin{thm}
\label{t3} Let $\ee>0$ and $\fff\subseteq [\ccc]^{<\omega}$ be an
arbitrary $\ee$-filling family. Then for every $P\subseteq \ccc$
perfect there exists $Q\subseteq P$ perfect such that for every
$R\subseteq Q$ perfect and every $k\geq 1$ the set $\fff\cap
[R]^k$ is dense in $[R]^k$.
\end{thm}
\noindent The proof of Theorem \ref{t3} follows the arguments of
the proof of Theorem \ref{t1}. The only severe change is that of
the notion of a regular dyadic the decides for $\fff$.
Specifically, Definition \ref{d1} is modified as follows.
\begin{defn}
Let $\fff\subseteq [\ccc]^{<\omega}$ and $T=(t_s)_{s\in\ct}$ be a
regular dyadic subtree of $\ct$. We say that the tree $T$ weakly
decides for $\fff$ if for every $n\in\omega$, every $0\leq d\leq
2^n-1$ and every $F=\{s_0\prec ...\prec s_d\}\subseteq 2^n$ we
have that one of the the following (mutually exclusive)
alternatives holds.
\begin{enumerate}
\item[(1)] Either $([\hat{T}]_{t_{s_0}}\times ... \times
[\hat{T}]_{t_{s_d}}) \cap\fff=\varnothing$, or \item[(2)] for
every $i=0,...,d$ and every $Q_i\subseteq [\hat{T}]_{t_{s_i}}$
perfect we have $(Q_0\times ...\times Q_d)\cap
\fff\neq\varnothing$.
\end{enumerate}
In the case where alternative (2) holds, then we say that $F$ is
weakly trapped in $\fff$.
\end{defn}
\noindent It can be easily checked that the arguments of the
proofs of Lemmas \ref{l1}, \ref{l2} and \ref{l3} can be carried
out using the above definitions, yielding the proof of Theorem
\ref{t3}.
\end{rem}

%--------------Weaker versions of density-------------------%

\section{Families of weaker density}

This section is devoted to the proof of Theorem B stated in the
introduction. For the convenience of the reader, let us present
the example of Fremlin which provides closed hereditary
families over $\ccc$ (of weaker density) for which Theorem
\ref{t1} is not valid.
\begin{examp}
\label{ex1} Let $f:\omega\to \omega$ be any function such that
$n\geq f(n)>0$ for all $n\geq 1$ and $\lim \frac{f(n)}{n}=0$. Then
there exists a family $\fff\subseteq [\ccc]^{<\omega}$ such that
the following hold.
\begin{enumerate}
\item[(1)] $\fff$ is closed in $[\ccc]^{<\omega}$ and hereditary.
\item[(2)] $d_\fff(n)\geq f(n)$ for all $n\geq 1$. \item[(3)]
There does not exist $A\subseteq\ccc$ infinite with
$[A]^{<\omega}\subseteq\fff$.
\end{enumerate}
Indeed, we can chose a strictly increasing sequence
$(n_k)_{k\in\omega}$ such that $n_0=1$ and $\sup_{i\geq n_k}
\frac{f(i)}{i} \leq \frac{1}{2^k}$ for all $k\geq 1$. We set
\[ \fff=\bigcup_{k\in\omega} \bigcup_{t\in 2^k} \Big\{ G: G\subseteq \ccc_{t}
\text{ and } |G|\leq \big\lceil n_{k+1}/2^k\big\rceil \Big\}.
\] It is easy to see that (1) and (3) are satisfied. To verify
(2), let $F\subseteq \ccc$ with $|F|=n$. Let $k\in\omega$ be such
that $n_k\leq n< n_{k+1}$. Then $F$ is partitioned in $\{ F\cap
\ccc_t\}_{t\in 2^k}$. There exists $t_0\in 2^k$ such that $|F\cap
\ccc_{t_0}|\geq \lceil n/2^k\rceil$. Let $G$ be any subset of
$F\cap\ccc_{t_0}$ with $|G|=\lceil n/2^k\rceil$. By the definition
of $\fff$ and the fact that $n<n_{k+1}$, we see that $G\in\fff$.
As $n\geq n_k$, we have $\frac{f(n)}{n}\leq \frac{1}{2^k}$ and so
$f(n)\leq \lceil n/2^k\rceil\leq d_\fff(n)$.
\end{examp}
Let us pass now to the proof of the main result of this section
(Theorem \ref{t2} below). Observe that for every $P\subseteq \ccc$
perfect and every $k$-type $\tau$, the set $[P]^k_{\tau}$ is
non-empty. We will need a finite version of this fact. To this
end, we make the following definitions.
\begin{defn}
Let $n\geq 1$. A finite subtree $T$ of the Cantor tree
$2^{<\omega}$ is said to be $n$-increasing if $T$ can be written
in the form $T=(t_s)_{s\in 2^{<n}}$ such that for all $s_1, s_2\in
2^{<n}$ the following are satisfied.
\begin{enumerate}
\item[(1)] $t_{s_1}\sqsubset t_{s_2}$ (respectively $t_{s_1}\prec
t_{s_2}$) if and only if $s_1\sqsubset s_2$ (respectively
$s_1\prec s_2$). \item[(2)] If $\ell(s_1)=\ell(s_2)$ and $s_1\prec
s_2$, then $\ell(t_{s_1})<\ell(t_{s_2})$. \item[(3)] If
$\ell(s_1)<\ell(s_2)$, then $\ell(t_{s_1})<\ell(t_{s_2})$.
\end{enumerate}
\end{defn}
\begin{defn}
A subset $F\subseteq \ccc$ with $|F|=2^n$ is said to be
$2^n$-increasing if the set $\spl(T_F)$ of splitting nodes of
$T_F$ forms an $n$-increasing subtree of $2^{<\omega}$. The set of
all $2^n$-increasing subsets of $\ccc$ will be denoted by $\incr$.
\end{defn}
It is easy to see that if $F$ is a $2^n$-increasing subset of
$\ccc$, then $T_F$ is a skew subtree of $\ct$. The class of
increasing subsets of $\ccc$ has the following stability property.
\begin{lem}
\label{l4} Let $n\geq 2$ and $k\geq 1$ be such that $2^n\geq n^k$.
Then for every $F\in \incr$ and every $G\subseteq F$ with $|G|\geq
n^k$ there exists $H\subseteq G$ with $H\in [\ccc]^{2^k}_{\vartriangle}$.
\end{lem}
\begin{proof}
Let $\spl(T_F)=(t_s)_{s\in 2^{<n}}$ be the set of splitting nodes
of the tree $T_F$. By our assumption, it is $n$-increasing. For
every $0\leq j\leq n-1$, we let $L_F(j)=\{ t_s\in \spl(T_F): s\in
2^j\}$. By the definition of $n$-increasing subtrees, the set
$L_F(j)$ is the $j$-level of $\spl(T_F)$ and so it is an antichain
of $\ct$. Let also $\spl(T_G)$ be the set of splitting nodes of
the tree $T_G$. Clearly $\spl(T_G)$ is a subset of $\spl(T_F)$.

Inductively, for every $0\leq m\leq k-1$ we shall construct
$j_m\in\omega$ and a subset $A_{m}$ of $\ct$ such that the
following are satisfied.
\begin{enumerate}
\item[(1)] $0\leq j_m \leq n-1$ and if $m_1<m_2$, then
$j_{m_1}>j_{m_2}$. \item[(2)] $2^{j_m}\geq n^{k-m-1}$. \item[(3)]
$A_{m}\subseteq L_F(j_m)$ and $|A_{m}|\geq n^{k-m-1}$. \item[(4)]
If $0\leq m_1<m_2\leq k-1$, then $A_{m_2}$ is a subset of
$\spl(\hat{A}_{m_1})$. \item[(5)] For all $0\leq m\leq k-1$, we
have $A_m\subseteq \spl(T_G)$.
\end{enumerate}
We start the construction. Notice  that the family $\{
\spl(T_G)\cap L_F(j)\}_{j=0}^{n-1}$ forms a partition of
$\spl(T_G)$. Since $|\spl(T_G)|=|G|-1\geq n^k-1$ there exists
$l\in\{0, ..., n-1\}$ such that $|\spl(T_G)\cap L_F(l)|\geq
n^{k-1}$. Notice that $|L_F(l)|=2^l\geq n^{k-1}$. We set $j_0=l$
and $A_0=\spl(T_G)\cap L_F(l)$. Then conditions (2), (3) and (5)
satisfied. This completes the first step of the inductive
construction. As $A_0$ is an antichain, being a subset of
$L_F(j_0)$, we have that $|\spl(\hat{A}_0)|=|A_0|-1\geq
n^{k-1}-1$. As in the first step, we notice that the family $\{
\spl(\hat{A}_0)\cap L_F(j)\}_{j=0}^{j_0-1}$ forms a partition of
$\spl(\hat{A}_0)$. Hence, there exists $l'\in\{ 0, ..., j_0-1\}$
such that $|\spl(\hat{A}_0)\cap L_F(l')|\geq n^{k-2}$. We set
$j_1=l'$ and $A_1=\spl(\hat{A}_0)\cap L_F(l')$. We proceed
similarly.

We isolate the crucial properties established by the above
construction.
\begin{enumerate}
\item[(P1)] For every $1\leq m\leq k-1$ and every $w\in A_m$, the
node $w$ has at least two successors in $A_{m-1}$.
\item[(P2)] For every $0\leq m \leq k-1$, if $w_1, w_2\in A_m$ with
$w_1\prec w_2$, then $\ell(w_1)<\ell(w_2)$.
\item[(P3)] For every $0\leq m_1< m_2 \leq k-1$, if $w_1 \in A_{m_1}$
and $w_2 \in A_{m_2}$, then $\ell(w_1)>\ell(w_2)$.
\end{enumerate}
Property (P1) follows by (4) of the construction while properties
(P2) and (P3) follow by (3) and (1) above and the fact that
$\spl(T_F)$ is $n$-increasing.

Using (P1)-(P3) and starting from a node in $A_{k-1}$ we construct
a $k$-increasing subtree $T=(w_s)_{s\in 2^{<k}}$ which is, by
condition (5) of the inductive construction, a subset of
$\spl(T_G)$. This clearly implies the lemma.
\end{proof}
\begin{lem}
\label{l5} Let $k\geq 1$ and $H\in [\ccc]^{2^k}_\vartriangle$.
Then for every $(k+1)$-type $\tau$ there exists $I\subseteq H$ of
type $\tau$.
\end{lem}
\begin{proof}
By induction on $k$. If $k=1$, then the result is trivial since we
can set $I=H$. Suppose that the result holds for some $k\geq 1$.
Let $H\in [\ccc]^{2^{k+1}}_\vartriangle$ and $\tau:\{1, ..., k+1\}
\to \omega$ be a $(k+2)$-type. Write $H$ in lexicographically
increasing order as $H=\{ y_0<...<y_{2^{k+1}-1}\}$ and put $E=\{
y_i: 0\leq i<2^{k+1}, i \text{ even}\}$. Let
$\spl(T_H)=(t_s)_{s\in 2^{<k+1}}$. It is easy to see that
$\spl(T_{E})=(t_s)_{s\in 2^{<k}}$ and so $E\in
[\ccc]^{2^k}_\vartriangle$. Let $\tau'=\tau|_{\{1,...,k\}}$. Then
$\tau'$ is a $(k+1)$-type. By our inductive assumption, there
exists $I'\subseteq E$ of type $\tau'$. There exists $\{
i_0<...<i_k\}\subseteq \{0,...,2^k-1\}$ such that $I'=\{
y_{2i_0}<...<y_{2i_k}\}$. We let $I=I'\cup \{
y_{2i_{\tau(k+1)}+1}\}$. Then $I\subseteq G$ and it is easy to
check that $I$ is of type $\tau$.
\end{proof}
\begin{lem}
\label{l6} Let $\fff\subseteq [\ccc]^{<\omega}$ hereditary, $n\geq
2$ and $k\geq 2$ be such that $d_\fff(2^n)\geq n^{k-1}$. Then for
every $P\subseteq \ccc$ perfect and every $k$-type $\tau$ there
exists $I\in\fff\cap [P]_\tau^k$.
\end{lem}
\begin{proof}
As $P$ is perfect, there exists a $2^n$-increasing subset $F$ of
$P$. Since $d_\fff(2^n)\geq n^{k-1}$, there exists $G\subseteq F$
with $G\in\fff$ and $|G|\geq n^{k-1}$. Notice that $2^n\geq
d_\fff(2^n)\geq n^{k-1}$. Hence, by Lemma \ref{l4}, there exists
$H\subseteq G$ which is $2^{k-1}$-increasing. By Lemma \ref{l5},
there exists $I\subseteq H$ of type $\tau$. As $I\subseteq
H\subseteq G\in \fff$ and $\fff$ is hereditary, the result
follows.
\end{proof}
We are ready to state and prove the main result of this section.
To this end, we recall A. Blass' theorem \cite{B} on partitions of
$[\ccc]^k$, which states that if $U$ is open subset of $[\ccc]^k$
and $\tau$ is a $k$-type, then there exists $P\subseteq\ccc$
perfect (which is the body of a skew tree) such that either
$[P]^k_{\tau}\subseteq U$ or $[P]^k_{\tau}\cap U=\varnothing$.
\begin{thm}
\label{t2} Let $\fff\subseteq [\ccc]^{<\omega}$ hereditary.
\begin{enumerate}
\item[(1)] Let $n\geq 2$ and $k\geq 1$ be such that
$d_\fff(2^n)\geq n^{k-1}$. If $\fff\cap [\ccc]^k$ has the Baire
property, then there exists $P\subseteq \ccc$ perfect such that
$[P]^k\subseteq\fff$. \item[(2)] Assume that $\fff$ has the Baire
property in $[\ccc]^{<\omega}$ and satisfies
\begin{equation*}
(\ast) \ \ \ \ \limsup \frac{\log_2 d_\fff(2^n)}{\log_2
n}=+\infty.
\end{equation*}
Then for every $k\geq 1$ there exists $P\subseteq \ccc$ perfect
such that $[P]^k\subseteq\fff$.
\end{enumerate}
\end{thm}
\begin{proof}
We argue first for part (1). If $k=1$ the result is trivial.
So let $k\geq 2$ and assume that $\fff$ has the Baire property
in $[\ccc]^k$. By a classical result of J. Mycielski (see
\cite{Kechris}) and by passing to a perfect subset of $\ccc$,
we may assume that $\fff\cap [\ccc]^k$ is open. Fix a $k$-type
$\tau$. By A. Blass' theorem there exists $P\subseteq \ccc$
perfect such that $[P]^k_{\tau}$ either is included in $\fff$
or is disjoint from $\fff$. The second alternative is impossible
by Lemma \ref{l6}. So the result follows by a finite exhaustion
argument over all possible $k$-types. Part (2) follows from
part (1) by a direct computation.
\end{proof}
\begin{rem}
We do not know whether equation $(\ast)$ in Theorem \ref{t2}(2) is
the optimal one. We notice, however, that the conclusion of part
(2) of Theorem \ref{t2} is not valid if we only assume that $\lim
d_\fff(n)=+\infty$. For instance, let $\fff$ be the union of all
strongly increasing and strongly decreasing finite subsets of
$\ccc$ (recall that a subset $\{x_0<... <x_k\}$ of $\ccc$ is said
to be strongly increasing if $\ell(x_i\wedge
x_{i+1})<\ell(x_{i+1}\wedge x_{i+2})$ for all $i\in \{0,...,
k-2\}$ -- a strongly decreasing subset of $\ccc$ is similarly
defined). Then $\fff$ is closed in $[\ccc]^{<\omega}$ and it is
easy to verify that $\lim d_\fff(n)= +\infty$. However, for every
$k\geq 4$ there does not exist a perfect subset $P$ of $\ccc$ with
$[P]^k\subseteq \fff$.
\end{rem}
\noindent \textbf{Consequences.} We start with the following
proposition which shows that the families presented in Example
\ref{ex1} are essentially the only ones within $C$-measurable
hereditary families which satisfy $(\ast)$.
\begin{prop}
\label{p1} Let $\fff\subseteq [\ccc]^{<\omega}$ hereditary. Assume
that $\fff$ is $C$-measurable in $[\ccc]^{<\omega}$ and satisfies
equation $(\ast)$ of Theorem \ref{t2}. Then for every
$g:\omega\to\omega$ and every $P\subseteq\ccc$ perfect there
exists a regular dyadic subtree $T=(t_s)_{s\in\ct}$ with
$[\hat{T}]\subseteq P$ and such that $\mathcal{G}\subseteq \fff$,
where $\mathcal{G}=\bigcup_{k\in\omega} \bigcup_{s\in 2^k} \big\{
G: G\subseteq [\hat{T}]_{t_s} \text{ and } |G|\leq g(k) \big\}$.
\end{prop}
\begin{proof}
By Theorem \ref{t2} and our assumptions, we have that for every
$P\subseteq \ccc$ perfect and every $m\geq 1$ there exists
$Q\subseteq P$ perfect with $[Q]^m\subseteq\fff$. Hence, arguing
as in Lemma \ref{l1}, we may construct a regular dyadic subtree
$T=(t_s)_{s\in\ct}$ and a family $(P^s)_{s\in\ct}$ of perfect
subsets of $P$ such that $t_\varnothing=\varnothing$ and moreover
the following hold.
\begin{enumerate}
\item[(i)] For every $k\in\omega$, every $s\in 2^k$ and every
$i\in\{0,1\}$, $P^{s^\con i}\subseteq P^s\cap \ccc_{t_{s^\con
i}}$. \item[(ii)] For every $k\in\omega$ and every $s\in 2^k$,
$[P^s]^{g(k)}\subseteq\fff$.
\end{enumerate}
Clearly $T$ is as desired.
\end{proof}
We need to introduce some more terminology. Let
$f:\omega\to\omega$ be such that $n\geq f(n)>0$ for all $n\geq 1$.
Let also $\fff\subseteq [\ccc]^{<\omega}$ and $A\subseteq \ccc$.
We say that $\fff$ is $f$-\textit{filling} over $A$ if for every
$n\geq 1$ and every $F\subseteq A$ with $|F|=n$ there exists
$G\subseteq F$ with $G\in\fff$ and $|G|\geq f(n)$. We notice that
if $\fff\subseteq [\ccc]^{<\omega}$ is an arbitrary hereditary
family with $\lim d_\fff(n)=+\infty$, then for every $A\subseteq
\ccc$ infinite there exists $B\subseteq A$ countable such that
$\fff$ becomes $1/2$-filling over $B$ (this follows by an
application of Ramsey's theorem). Although, by Theorem \ref{t1},
this fact cannot be extended to perfect sets, it can be extended
for weaker versions of density as the following corollary
demonstrates.
\begin{cor} \label{c3} Let
$\fff\subseteq [\ccc]^{<\omega}$ be as in Proposition \ref{p1} and
$f:\omega\to\omega$ be such that $n\geq f(n)>0$ for all $n\geq 1$
and $\lim \frac{f(n)}{n}=0$. Then for every $P\subseteq\ccc$
perfect there exists $Q\subseteq P$ perfect such that $\fff$ is
$f$-filling over $Q$.
\end{cor}
\begin{proof}
We may select a strictly increasing sequence $(n_k)_{k\in\omega}$
such that $n_0=1$ and $\sup_{i\geq n_k}\frac{f(i)}{i}\leq
\frac{1}{2^k}$ for all $k\in\omega$. We define $g:\omega\to\omega$
by $g(k)=\big\lceil n_{k+1}/2^k \big\rceil$. Let
$T=(t_s)_{s\in\ct}$ be the regular dyadic subtree obtained by
Proposition \ref{p1} for the function $g$ and the given perfect
set $P$. Setting $Q=[\hat{T}]$ and arguing as in Example
\ref{ex1}, we can easily verify that $Q$ has all the desired
properties.
\end{proof}

%----------------Connections with Banach spaces---------------------%

\section{Connections with Banach spaces}

Theorem \ref{t1} has some Banach space theoretic implications
which we are about to describe. Let $\fff\subseteq
[\ccc]^{<\omega}$ hereditary with $[\ccc]^1\subseteq\fff$. For
every such family $\fff$ we define a Banach space $X_\fff$ as
follows. Let $c_{00}(\ccc)$ be the vector space of all real-valued
functions on $\ccc$ with finite support and denote by
$(e_x)_{x\in\ccc}$ the standard Hamel basis of $c_{00}(\ccc)$.
Then $X_\fff$ is the completion of $c_{00}(\ccc)$ under the norm
$\|\cdot\|_\fff$ defined by
\[ \Big\| \sum_{i=0}^n a_i e_{x_i} \Big\|_{\fff} =\sup \Big\{
\sum_{i\in F} |a_i|: \{ x_i:i\in F\}\in\fff \Big\}. \]  We recall
that a bounded sequence $(e_n)_n$ in a Banach space $E$ is called
\textit{Cesaro summable} if the sequence of averages
$\frac{e_0+...+e_{n-1}}{n}$ converges in norm. Under the above
terminology we have the following proposition.
\begin{prop}
\label{p2} Let $\fff\subseteq [\ccc]^{<\omega}$ be hereditary,
compact and such that $[\ccc]^1\subseteq\fff$. Assume that $\fff$
is $C$-measurable and $\lim d_\fff(n)=+\infty$. Then the following
hold.
\begin{enumerate}
\item[(1)] For every sequence $(x_i)_i$ in $\ccc$ there exists
$L\subseteq\omega$ infinite such that for every $N\subseteq L$
infinite the sequence $(e_{x_i})_{i\in N}$ is not Cesaro summable
in $X_\fff$.
\end{enumerate}
But on the other hand,
\begin{enumerate}
\item[(2)] for every $P\subseteq\ccc$ perfect, there exists
$(x_i)_i$ in $P$ such that the sequence $(e_{x_i})_i$ is Cesaro
summable in $X_\fff$.
\end{enumerate}
\end{prop}
\begin{proof}
(1) Let $(x_i)_i$ be a sequence in $\ccc$. As we have already
remarked, by the fact that $\fff$ is hereditary and $\lim
d_\fff(n)=+\infty$, there exists $L\subseteq \omega$ infinite such
that $\fff$ is $1/2$-filling over $\{ x_i:i\in L\}$. By the
definition of the norm of $X_\fff$, we see that for every
$F\subseteq L$ finite we have $\big\| \sum_{i\in F}
e_{x_i}\big\|_{\fff} \geq \frac{|F|}{2}$. This clearly implies
that for every $N\subseteq L$ infinite the sequence
$(e_{x_i})_{i\in N}$ is not Cesaro summable in $X_\fff$.\\
(2) Let $P\subseteq\ccc$ perfect. By our assumptions, Lemma
\ref{l1} can be applied. Hence, there exists a regular dyadic
subtree $T=(t_s)_{s\in\ct}$ that decides for $\fff$ and
$[\hat{T}]\subseteq P$.

Let $Z$ be the set of all eventually zero sequences in $\ccc$. We
enumerate $Z$ as $(z_i)_{i\in\omega}$ as follows. For every
$i\in\omega$ let $z_i$ be the unique element of $Z$ satisfying
$i=\sum_{k\in\omega} z_i(k) 2^k$. By the uniqueness of the dyadic
representation of every natural number, we have that if $i\neq j$,
then $z_i\neq z_j$ and moreover, if $n,i,j\in\omega$ are such that
$i,j<2^n$, then $z_i|n\neq z_j|n$.

For every $i\in\omega$ define $x_i=\bigcup_{k\in\omega}
t_{z_i|k}\in [\hat{T}]$. We claim that $(x_i)_{i\in\omega}$ is the
desired sequence. To this end, for every $n\in\omega$ let $s\in
2^n$ and put $l_n=\ell(t_s)$ (as $T$ is regular dyadic $l_n$ is
well-defined and independent of the choice of $s$). By the above
mentioned property of the sequence $(z_i)_i$, for every
$i,n\in\omega$ with $i<2^n$ we have that
\begin{equation}
\label{e1} |\{ x_0|l_{n},...,x_i|l_{n}\}|=|\{x_0,...,x_i\}|=i+1.
\end{equation}
For every $n\in\omega$ define
\[ M_n=\max\big\{ |F|: F\subseteq 2^n \text{ and } F \text{ is trapped in }
\fff\big\}. \] By (\ref{e1}) and the fact that the tree $T$
decides for $\fff$, for every $i<2^n$ we get that
\[ \max \big\{ |G|: G\subseteq\{ x_0,...,x_{i}\} \text{ and } G\in\fff
\big\}\leq M_{n}. \] Let $i,n\geq 1$ with $2^{n-1}\leq i<2^n$.
Then
\begin{equation}
\label{e2} \Big\| \frac{1}{i+1}\sum_{k=0}^{i} e_{x_k}
\Big\|_{\fff} \leq \frac{M_{n}}{i+1}\leq \frac{M_{n}}{2^{n-1}}=
2\frac{M_{n}}{2^{n}}.
\end{equation}
Finally notice that
\[ \lim \frac{M_n}{2^n}=0. \]
For if not, arguing as in the proof of Lemma \ref{l3}, we would
have that there exists $R\subseteq [\hat{T}]$ perfect with
$[R]^{<\omega}\subseteq\fff$, contradicting the fact that $\fff$
is compact. Hence, by (\ref{e2}), we have
\[ \frac{1}{i+1}\sum_{k=0}^{i} e_{x_k}\to 0\]
and the proof is completed.
\end{proof}
\begin{rem}
(a) Part (2) of Proposition \ref{p2} can also be derived by
Theorem 3A in \cite{F}, taking into account that every
$C$-measurable, hereditary and compact family $\fff$ is not
$\ee$-filling for every $\ee>0$. For completeness we have
included a proof in the present setting.\\
(b) We notice that the fact that every subsequence of the sequence
$(x_i)_{i\in L}$, obtained in part (1) of Proposition \ref{p2}, is
not Cesaro summable, is expected by the Erd\"{o}s-Magidor theorem
\cite{EM} (see also \cite{AT}). \\
(c) We notice that under the assumptions of Proposition \ref{p2},
for every $P\subseteq\ccc$ perfect there exists $Q\subseteq P$
perfect with the following property. If $(x_i)_i$ is a sequence in
$Q$ and the sequence $(e_{x_i})_i$ generates a spreading model
(see \cite{AT} for the definition), then this spreading model must
be $\ell_1$.
\end{rem}

%-------------------------------------------------------------------%
%                           Bibliography                            %
%-------------------------------------------------------------------%

\end{document}